\documentclass[11 pt]{article}
\usepackage{latexsym}
\usepackage{graphicx}
\usepackage{amsmath}
\usepackage{amsthm}
\usepackage{amssymb,amsfonts,enumerate}
\usepackage[pdfpagelabels=true,plainpages=false,colorlinks=true,linkcolor=blue,citecolor=blue,urlcolor=blue]{hyperref}
\newtheorem{theorem}{Theorem}[section]
\newtheorem{lemma}{Lemma}[section]

\newtheorem{definition}{Definition}[section]
\newtheorem{remark}{Remark}[section]
\newtheorem{example}{Example}[section]

\begin{document}
\begin{center}
{\Large{\bf Weighted $\beta\gamma$-summability of fuzzy functions of order $\theta$}}\\\vspace{.5cm}
Sarita Ojha and P. D. Srivastava\\
Department of Mathematics, Indian Institute of Technology,\\ Kharagpur 721302, India
\end{center}
\vspace{1cm}

%\section*{Abstract}
%Based on the concept of new type of statistical convergence defined by Aktuglu \cite{A}, we have introduced the weighted $\beta\gamma$ - statistical convergence of order $\theta$  in case of fuzzy functions and classified it into pointwise, uniform and equi-statistical convergence. We have checked some basic properties and then the convergence are investigated in terms of their $\alpha$-cuts. The interrelation among them are also established. We have also proved that continuity, boundedness etc are preserved in the equi-statistical sense under some suitable conditions, but not in pointwise sense.
%%\keywords{First keyword \and Second keyword \and More}

\section*{Abstract}
The concept of weighted $\beta\gamma$ - summability of order $\theta$ in case of fuzzy functions is introduced and classified into ordinary and absolute sense. Several inclusion relations among the sets are investigated. Also we have found some suitable conditions to get its relation with the generalized statistical convergence. Finally we have proved a generalized version of Tauberian theorem.

\section{Introduction}
Convergence of sequences, in classical or fuzzy sense, means that almost all
elements of the sequence have to belong to an arbitrary small neighborhood
of the limit. The aim to introduce the concept of statistical convergence is
to relax the above condition and to examine the convergence criterion only
for majority of elements. The recent trend in mathematics is to investigate the above type of convergence in a more generalized way and link it with summability theory for sequences of real as well as fuzzy numbers.\\
After introducing the concept of statistical convergence for scalar sequences by Fast \cite{F}, several development in this direction i.e. convergence along with the summability methods, for the sequences of real or complex numbers as well as for fuzzy numbers, can be enumerated as follows:
\begin{enumerate}[(i)]
\item $\lambda$ - statistical convergence of order $\beta$ \cite{At}, \cite{CB}, \cite{EAM}.
\item lacunary statistical convergence of order $\beta$ \cite{AG}, \cite{Kw1}, \cite{N}.
\item weighted statistical convergence \cite{A}, \cite{MM}.
\end{enumerate}
Recently Aktuglu \cite{A} has given a general type of convergence by defining $\beta\gamma$ - density and in the light of this he has defined statistical convergence. Motivated by Gong et al \cite{ZZZ},  Ojha and Srivastava \cite{SO1} has introduced weighted $\beta\gamma$ - statistical convergence for fuzzy functions by using the above density and given three types of classifications of it.\\
In the existing literature of summability method for fuzzy numbers, it is observed that the authors are dealing either with the fuzzy numbers or with the metric values of the fuzzy numbers. The main aim of introducing weighted $\beta\gamma$ - summability is to investigate these both type of definitions and their inter-relations. A generalized version of Tauberian theorem is also established. This summability method not only includes some well known matrix methods but also gives some non-regular matrix methods.

\section{Preliminaries}
A fuzzy real number $\hat{x}:\mathbb{R}\rightarrow [0,1]$ is a fuzzy set which is normal, fuzzy convex, upper semicontinuous and $[\hat{x}]_0=\{t\in \mathbb{R}:\hat{x}(t)>0\}$ is compact. Clearly, $\mathbb{R}$ is embedded in $L(R)$, the set of all fuzzy numbers, in the following way:\\
\noindent For each $r\in \mathbb{R}, \overline{r}\in L(R)$ is defined as,
\begin{center}
$\overline{r}(t) = \left\{
\begin{array}{c l}
  1, & t=r \\
  0, & t\neq r
\end{array}
\right.$\end{center}
\noindent For $0<\alpha\leq1$, $\alpha$-cut of $\hat{x}$ is defined by $[\hat{x}]_{(\alpha)}=\{t\in\mathbb{R}:\hat{x}(t)\geq\alpha\}=[(\hat{x})_{\alpha} ^-,(\hat{x})_{\alpha} ^+]$, a closed and bounded interval of $\mathbb{R}$. Now for any two fuzzy numbers $\hat{x},\hat{y}$, Matloka \cite{M} has proved that $L(R)$ is a complete metric space under the metric $d$ defined by
\begin{equation*}
d(\hat{x},\hat{y})=\sup\limits_{0\leq \alpha \leq1} \max\{|(\hat{x})_{\alpha} ^--(\hat{y})_{\alpha} ^-|,|(\hat{x})_{\alpha} ^+-(\hat{y})_{\alpha} ^+|\}
\end{equation*}
For any $\hat{x},\hat{y},\hat{z},\hat{w}\in L(R)$, the metric $d$ satisfies
\begin{enumerate}[(i)]
\item $d(c\hat{x},c\hat{y})=|c| d(\hat{x},\hat{y})$, $c\in\mathbb{R}$.
\item $d(\hat{x}+\hat{z},\hat{y}+\hat{z})=d(\hat{x},\hat{y})$.
\item $d(\hat{x}+\hat{z},\hat{y}+\hat{w})\leq d(\hat{x},\hat{y})+d(\hat{z},\hat{w})$.
\end{enumerate}

\noindent For any two fuzzy numbers $\hat{x},\hat{y}\in L(R)$, a partial order relation $\preceq$ is defined as
\begin{equation*}
\hat{x}\preceq \hat{y}  \ \Leftrightarrow \ (\hat{x})_{\alpha} ^- \leq (\hat{y})_{\alpha} ^-\ \mbox{ and } \ (\hat{x})_{\alpha} ^+\leq (\hat{y})_{\alpha} ^+
\end{equation*}
holds for each $\alpha\in[0,1]$.
\begin{lemma}
For the given fuzzy numbers $\hat{x},\hat{y}\in L(R)$, the following statements are equivalent:
\begin{enumerate}[(i)]
\item $d(\hat{x},\hat{y})\leq \varepsilon$.
\item $\hat{x}-\varepsilon \preceq \hat{y} \preceq \hat{x}+\varepsilon$.
\end{enumerate}
\end{lemma}

\begin{definition} [Statistical Convergence]
A sequence $x=(x_k)$ of real or complex numbers is said to be statistically convergent to a
number $l$ if for every $\varepsilon>0$,
\begin{equation*}
\lim\limits_{n\rightarrow\infty} \frac{1}{n}|\{k\leq n:|x_k-l|\geq
\varepsilon\}|=0
\end{equation*}
\end{definition}
\noindent where the vertical bars $|\cdot|$ indicate the number of elements in the enclosed set.

\noindent Many authors have generalized the definition of statistical convergence for the real or complex numbers. The most recent generalization in this direction is the idea of $\beta\gamma$ - statistical convergence which is introduced by Aktuglu \cite{A} as follows:
\begin{definition}
Let $(\beta_n),(\gamma_n)$ be two sequences of positive numbers s.t.
\begin{enumerate}[(i)]
\item $(\beta_n)$, $(\gamma_n)$ are both non-decreasing
\item $\gamma_n\geq \beta_n$
\item $\gamma_n-\beta_n\to\infty$ as $n\to\infty$
\end{enumerate}
Then a sequence $(x_k)$ of real numbers is said to be $\beta\gamma$ - statistically convergent of order $\theta$ to $L$ if for every $\varepsilon>0$,
\begin{equation*}
\lim\limits_{n\to\infty} \frac{1}{(\gamma_n-\beta_n+1)^{\theta}} \Big|\Big\{k\in[\beta_n,\gamma_n]: |x_k-L|\geq\varepsilon\Big\}\Big|= 0.
\end{equation*}
\end{definition}

\begin{remark}
The above definition includes the following cases:
\begin{enumerate}[(i)]
\item By choosing $\beta_n=1,\gamma_n=n$, this $\beta\gamma$ - convergence coincides with the usual statistical convergence as discussed by Fast \cite{F}.
\item Let $(\lambda_n)$ be a non-decreasing sequence of positive real numbers tending to $\infty$ such that $\lambda_1=1,\lambda_{n+1}\leq \lambda_n+1$ for all $n$. Then by setting $\beta_n=n-\lambda_n+1$ and $\gamma_n=n$, $\beta\gamma$ - convergence coincides with the concept of $\lambda$ - statistical convergence of order $\theta$ as as discussed by Colak et al \cite{CB}.
\item Choosing $\beta_r=k_{r-1}+1,\gamma_r=k_r$ where $(k_r)$ is an increasing integer sequence with $k_0=0$ and $h_r=k_r-k_{r-1}$, then the $\beta\gamma$ - convergence coincides with lacunary statistical convergence of order $\theta$ given by Aktuglu \cite{AG}.
\end{enumerate}
\end{remark}

\noindent Ojha and Srivastava \cite{SO1} have introduced the concept of weighted $\beta\gamma$ - statistical convergence in case of fuzzy functions with the modification given by Ghoshal \cite{G1} for the weighted convergence definition as follows: \\
Let $\theta$ be a real number such that $0<\theta\leq1$ and $t=(t_k)$ be a sequence of non-negative real numbers such that $\liminf\limits_k t_k>0$ and suppose
\begin{equation*}
T_{\beta\gamma (n)}=\sum\limits_{k\in  [\beta_n,\gamma_n]} t_k,\ n\in\mathbb{N}
\end{equation*}

\begin{definition}
A sequence $\hat{f_k}:[a,b]\to L(R)$ of fuzzy functions is said to be weighted $\beta\gamma$ - pointwise statistical convergent of order $\theta$ to $\hat{f}$ if for every $\varepsilon>0$ and for each $x\in[a,b]$,
\begin{equation*}
\lim\limits_{n\to\infty} \frac{1}{T_{\beta\gamma (n)} ^{\theta}} \Big|\Big\{k\leq T_{\beta\gamma (n)}: t_kd(\hat{f_k}(x),\hat{f}(x))\geq\varepsilon\Big\}_x\Big|=0
\end{equation*}
where $\{\cdot\}_x$ denotes that the set depends on the point $x$. Denote the set of all weighted $\beta\gamma$ - pointwise statistical convergent sequence of fuzzy functions of order $\theta$ by $SP_{\beta\gamma} ^{\theta}(t)$.
\end{definition}

%\begin{definition}
%A sequence $\hat{f_k}:[a,b]\to L(R)$ of fuzzy functions is said to be weighted $\beta\gamma$ - uniformly statistical convergent of order $\theta$ to $\hat{f}$ if for every $\varepsilon>0$ and for all $x\in[a,b]$,
%\begin{equation*}
%\lim\limits_{n\to\infty} \frac{1}{T_{\beta\gamma (n)} ^{\theta}} \Big|\Big\{k\leq T_{\beta\gamma (n)}: t_kd(\hat{f_k}(x),\hat{f}(x))\geq\varepsilon\Big\}\Big|=0
%\end{equation*}
%Denote the set of all weighted $\beta\gamma$ - uniform statistical convergent of fuzzy functions of order $\theta$ by $SU_{\beta\gamma} ^{\theta}(t)$.
%\end{definition}
%\vspace{1cm}

\begin{definition}[Weighted strong Cesaro summability] \cite{G}\\
Let $(t_n)$ be a sequence of nonnegative real numbers such that $t_1>0$, $T_n=t_1+t_2+\cdots+t_n$ and $T_n\to\infty$ as $n\to\infty$. Then a sequence of real numbers $(x_n)$ is said to be Weighted strong Cesaro convergence to a real number $x$ if
\begin{equation*}
\lim\limits_{n\rightarrow\infty} \frac{1}{T_n}\sum\limits_{k=1} ^n t_k|x_k-x|=0.
\end{equation*}
\end{definition}
Use of various type of summability methods has motivated us to define the notion of weighted $\beta\gamma$ - summability of order $\theta$ for fuzzy functions. We have also classified the method into ordinary and absolute sense and obtained some interesting results on it.

\section{Absolutely weighted $\beta\gamma$-summability of order $\theta$}
Let $t=(t_n)$ be a sequence of positive real numbers such that $t_1>0$ and
\begin{equation}
T_{\beta\gamma(n)}=\sum\limits_{k\in[\beta_n,\gamma_n]} t_k \to\infty
\end{equation}
as $n\to\infty$. Let $\hat{f},\hat{f}_k:[a,b]\to\mathbb{R}$ be a sequence of fuzzy functions and $\theta$ belongs to $(0,1]$. For each $x\in[a,b]$, consider
\begin{equation*}
s_n(x)=\frac{1}{T_{\beta\gamma(n)} ^{\theta}} \sum\limits_{k\in[\beta_n,\gamma_n]} t_kd(\hat{f_k}(x),\hat{f}(x))
\end{equation*}
\begin{definition}
A sequence $(\hat{f_k})$ of fuzzy functions is said to be {\bf absolutely weighted $\beta\gamma$ - summable} to a fuzzy function $\hat{f}$ of order $\theta$ if
\begin{equation*}
s_n(x)\to 0\ \mbox{ for every $x\in[a,b]$}
\end{equation*}
We denote this convergence as $N_{\beta\gamma} ^{\theta} (t)$. In case of $\theta=1$, the set $N_{\beta\gamma} ^{\theta} (t)$ reduces to $N_{\beta\gamma} (t)$.
\end{definition}

\begin{remark}[Particular cases]
For $t_n=1$ and
\begin{enumerate}[(i)]
\item $\beta_n=1,\gamma_n=n, \theta=1$, $N_{\beta\gamma} ^{\theta} (t)$ reduces to Cesaro summability which is defined for fuzzy numbers by Kwon \cite{Kw}.
\item $\beta_n=n-\lambda_n+1,\gamma_n=n$, where $\lambda$ is defined as in Remark 1(ii), then $N_{\beta\gamma} ^{\theta} (t)$ gives the class of fuzzy numbers which is $\lambda$ -summable \cite{S}, $\lambda$ -summable of order $\theta$ \cite{At} and the class of $\lambda$ - summable of order $\theta$ of fuzzy functions by \cite{SO}.
\item $\beta_n=k_{n-1}+1,\gamma_n=k_n$ where $(k_n)$ is defined as in Remark 1(iii), then $N_{\beta\gamma} ^{\theta} (t)$ is the class of all lacunary summable sequence which is due to Kwon et al \cite{Kw1} for fuzzy numbers.
\end{enumerate}
\end{remark}

\begin{theorem}
Let $\hat{f}_k:[a,b]\to L(R)$ and $\hat{g}_k:[a,b]\to L(R)$ be two sequence of fuzzy functions. Then
\begin{enumerate}[(i)]
\item If $(\hat{f}_k),(\hat{g}_k)\in N_{\beta\gamma} ^{\theta} (t)$, then
$(\hat{f}_k+\hat{g}_k),(c\hat{f}_k)\in N_{\beta\gamma} ^{\theta} (t)$ for $c\in\mathbb{R}$.
\item If $0<\theta\leq\delta\leq1$, then $N_{\beta\gamma} ^{\theta} (t)\subseteq N_{\beta\gamma} ^{\delta} (t)$. The inclusion is strict for some $\theta,\delta$ where $\theta<\delta$.
\end{enumerate}
\end{theorem}
\begin{proof}
\begin{enumerate}[(i)]
\item This part follows from the following relations:
\begin{eqnarray*}
d(\hat{f}_k(x)+\hat{g}_k(x),\hat{f}(x)+\hat{g}(x)) &\leq& d(\hat{f}_k(x),\hat{f}(x))+d(\hat{g}_k(x),\hat{g}(x))\\
\mbox{and}\  \ d(c\hat{f}_k(x),c\hat{f}(x)) &=& |c|d(\hat{f}_k(x),\hat{f}(x))
\end{eqnarray*}
for each $x\in[a,b]$. Since $(t_k)$ are positive real numbers, so the result follows.
\item Since $T_{\beta\gamma(n)}\to\infty$ as $n\to\infty$. So for some fixed $n_0\in\mathbb{N}$, $T_{\beta\gamma(n)}\geq1$ for all $n\geq n_0$ and hence for $\theta\leq\delta$, we have $T_{\beta\gamma(n)} ^{\theta}\leq T_{\beta\gamma(n)} ^{\delta}$ for all $n\geq n_0$. Therefore
    \begin{equation*}
    \frac{1}{T_{\beta\gamma(n)} ^{\delta}} \sum\limits_{k\in[\beta_n,\gamma_n]} t_kd(\hat{f_k}(x),\hat{f}(x))\leq \frac{1}{T_{\beta\gamma(n)} ^{\theta}} \sum\limits_{k\in[\beta_n,\gamma_n]} t_kd(\hat{f_k}(x),\hat{f}(x))\ \mbox{for all } n\geq n_0
    \end{equation*}
    Which implies $N_{\beta\gamma} ^{\theta} (t)\subset N_{\beta\gamma} ^{\delta} (t)$. To show the inequality is strict, consider the following example.
    \begin{example}
    Construct the sequence of functions $(\hat{f}_k)$ as:
    \begin{equation*}
    \hat{f}_k(x)=\begin{cases}
    \bar{0}, & k \mbox{ is a square number}\\
    \hat{f}(x), & \mbox{otherwise}
    \end{cases}
    \end{equation*}
    for each $x\in[a,b]$ and $\hat{f}:[a,b]\to L(R)$ be a bounded fuzzy function i.e. for $x\in[a,b]$, $d(\hat{f}(x),\bar{0})\leq M$ (say). Let $t_k=1$ for all $k$. Then we have
    \begin{eqnarray*}
    \frac{1}{(\gamma_n-\beta_n+1) ^{\theta}} \sum\limits_{k\in[\beta_n,\gamma_n]} t_kd(\hat{f_k}(x),\hat{f}(x)) = \frac{\Big[\sqrt{(\gamma_n-\beta_n+1)}\Big]}{(\gamma_n-\beta_n+1) ^{\theta}}d(\hat{f}(x),\bar{0})
    \end{eqnarray*}
    where $[ y ]$ is the greatest integer function less than or equal to $y$. Now
      \begin{eqnarray*}
       \frac{\sqrt{(\gamma_n-\beta_n+1)}-1}{(\gamma_n-\beta_n+1) ^{\theta}}d(\hat{f}(x),\bar{0}) &\leq& \frac{\Big[\sqrt{(\gamma_n-\beta_n+1)}\Big]}{(\gamma_n-\beta_n+1) ^{\theta}}d(\hat{f}(x),\bar{0})\\ && \leq  \frac{\sqrt{(\gamma_n-\beta_n+1)}}{(\gamma_n-\beta_n+1) ^{\theta}}d(\hat{f}(x),\bar{0})
    \end{eqnarray*}
    So from the above inequality, it is clear that $s_n(x)\to 0$ for choosing $\theta>1/2$ and $s_n(x)$ is divergent for $\theta<1/2$ as $\hat{f}(x)$ is bounded. So by choosing $\theta<1/2<\delta$, we can conclude that $(\hat{f_k}) \in N_{\beta\gamma} ^{\delta} (t)$ but $(\hat{f_k}) \notin N_{\beta\gamma} ^{\theta} (t)$.
    \end{example}
\end{enumerate}
This completes the proof.
\end{proof}

\begin{theorem}
The set $N_{\beta\gamma} (t)\cap B^F [a,b]$ is a closed subset of $B^F [a,b]$ where $B^F[a,b]$ is the set of all bounded fuzzy functions on $[a,b]$.
\end{theorem}
\begin{proof}
Let $(\hat{f} ^i)=(\hat{f} ^i _k)\in N_{\beta\gamma} (t)\cap B^F [a,b]$ such that for each $k\in\mathbb{N}$ and $x\in[a,b]$, $\hat{f} ^i _k (x)\to \hat{f} _k(x)$ as $i\to\infty$. Write $\hat{f}=(\hat{f}_k)$. We have to show $\hat{f}\in N_{\beta\gamma} (t)\cap B^F [a,b]$.\\
Since $(\hat{f} ^i)=(\hat{f} ^i _k)\in N_{\beta\gamma} (t)$, so there exists a fuzzy functions $\hat{g} _i :[a,b]\to L(R)$ and an integer $n_1$ such that for all $i,j,n\geq n_1$ and for all $x\in[a,b]$
\begin{eqnarray*}
\frac{1}{T_{\beta\gamma(n)}} \sum\limits_{k\in[\beta_n,\gamma_n]} t_kd(\hat{f} ^i _k(x),\hat{g} _i(x))<\frac{\varepsilon}{3}
\end{eqnarray*}
Further as $(\hat{f} ^i _k)$ converges to $\hat{f} _k(x)$, so it is Cauchy i.e.
\begin{eqnarray*}
d(\hat{f} ^i _k (x),\hat{f} ^j _k (x))<\frac{\varepsilon}{3}
\end{eqnarray*}
holds. So for all $n\geq n_1$, we have
\begin{eqnarray*}
d(\hat{g} _i(x),\hat{g} _j(x)) &\leq& d(\hat{f} ^i _k(x),\hat{g} _i(x))+d(\hat{f} ^i _k(x),\hat{f} ^j _k(x))+d(\hat{f} ^j _k(x),\hat{g} _j(x))\\
\mbox{i.e. }\ \frac{1}{T_{\beta\gamma(n)}} \sum\limits_{k\in[\beta_n,\gamma_n]} t_kd(\hat{g} _i(x),\hat{g} _j(x)) &\leq& \frac{1}{T_{\beta\gamma(n)}} \sum\limits_{k\in[\beta_n,\gamma_n]}t_k\Big[d(\hat{f} ^i _k(x),\hat{g} _i(x))+\frac{\varepsilon}{3}\\ && +d(\hat{f} ^j _k(x),\hat{g} _j(x))\Big]\\
\mbox{i.e. }\  d(\hat{g} _i(x),\hat{g} _j(x)) &\leq& \frac{1}{T_{\beta\gamma(n)}} \sum\limits_{k\in[\beta_n,\gamma_n]}t_k d(\hat{f} ^i _k(x),\hat{g} _i(x))+ \frac{\varepsilon}{3}\\ && + \frac{1}{T_{\beta\gamma(n)}}\sum\limits_{k\in[\beta_n,\gamma_n]}t_kd(\hat{f} ^j _k(x),\hat{g} _j(x))\\
&<& \varepsilon
\end{eqnarray*}
Thus $(\hat{g}_i(x))$ is a Cauchy sequence in $L(R)$ and $L(R)$ is complete, so $(\hat{g}_i(x))$ is convergent, say $\hat{g}_i(x)\to \hat{g}(x)$ for each $x\in[a,b]$ as $i\to\infty$. \\
So there exists an integer $n_2$ such that for all $i,k,n\geq n_2$ and $\forall \ x\in[a,b]$,
\begin{eqnarray*}
d(\hat{g} _k (x),\hat{g} (x)) &<& \frac{\varepsilon}{4}\\
d(\hat{f} ^i _k (x),\hat{f} _k (x)) &<& \frac{\varepsilon}{4} \ \mbox{ (by assumption)}\\
\mbox{and } \frac{1}{T_{\beta\gamma(n)}} \sum\limits_{k\in[\beta_n,\gamma_n]} t_kd(\hat{f} ^i _k (x),\hat{g} _i(x)) &<& \frac{\varepsilon}{2} \ \mbox{ (by assumption)}
\end{eqnarray*}
holds. Then for all $i,k,n\geq p$ and for all $x\in[a,b]$,
\begin{eqnarray*}
d(\hat{f} _k (x),\hat{g} (x)) &\leq& d(\hat{f} ^i _k (x),\hat{f} _k (x))+d(\hat{f} ^i _k (x),\hat{g} _i (x))+d(\hat{g} _i (x),\hat{g} (x))\\
&<& \frac{\varepsilon}{2} + d(\hat{f} ^i _k (x),\hat{g} _i (x))\\
\mbox{i.e. } \frac{1}{T_{\beta\gamma(n)}} \sum\limits_{k\in[\beta_n,\gamma_n]} t_kd(\hat{f} _k (x),\hat{g} (x)) &\leq& \frac{\varepsilon}{2} + \frac{1}{T_{\beta\gamma(n)}} \sum\limits_{k\in[\beta_n,\gamma_n]} t_kd(\hat{f} ^i _k (x),\hat{g} _i (x))<\varepsilon
\end{eqnarray*}
So we can conclude now, $(\hat{f} _k)\in N_{\beta\gamma} (t)\cap B^F [a,b]$ i.e. $N_{\beta\gamma} (t)\cap B^F [a,b]$ is a closed subset of $B^F [a,b]$.
\end{proof}

\begin{remark}
The above theorem is not true in case of $\theta\neq1$. To prove this, consider the following example.\\
Let $t_n=1$ for all $n$. Take the following sequences of fuzzy functions.
\begin{eqnarray*}
f^{(n)} _k(x) &=& \begin{cases}
\bar{0}, & k\leq n \ \mbox{and}\ k \ \mbox{ is a square number}\\
\hat{f}(x), & \mbox{otherwise}
\end{cases}\\
f_k(x) &=& \begin{cases}
\bar{0}, & k \mbox{ is a square number}\\
\hat{f}(x), & \mbox{otherwise}
\end{cases}
\end{eqnarray*}
where $\hat{f}(x)$ is bounded on $[a,b]$. It is clear that $f^{(n)} _k(x)\to f_k(x)$ for each $x\in[a,b]$ and also $\Big(f^{(n)} _k(x)\Big)\in N_{\beta\gamma} ^{\theta} (t)$ to the function $f_k(x)$ for any choice $(\beta_n),(\gamma_n)$ and $\theta$. But from the example 3.1, it is clear that $\Big(f_k(x)\Big)\notin N_{\beta\gamma} ^{\theta} (t)$ for $\theta<1/2$.
\end{remark}

\subsection{Inclusion between weighted $\beta\gamma$ - pointwise statistical convergence and absolutely weighted $\beta\gamma$ - summability}
We know that summability methods eg. Cesaro, $\lambda$ - statistical, lacunary summability etc imply the corresponding type of statistical convergence in real or in fuzzy number systems. So it is interesting to know the behavior of the weighted $\beta\gamma$ - summability for the case of fuzzy functions of order $\theta$ with the corresponding type of weighted $\beta\gamma$ - statistical convergence of fuzzy functions of order $\theta$. To investigate the inclusions by our definitions, we assume
$$0<\theta_1\leq \theta_2\leq 1 \  \mbox{ and } \  \liminf\limits_{n\to\infty} t_n>0$$
Now we prove the following theorems.
\begin{theorem}
Let $\liminf\limits_{n\to\infty} \frac{\gamma_n}{T_{\beta\gamma(n)}}>1$ and $0< \beta_n\leq1$. Then if a sequence of fuzzy functions is absolutely weighted $\beta\gamma$ - summable of order $\theta_1$, it is also weighted $\beta\gamma$ - statistically convergent of order $\theta_2$ i.e. $N_{\beta\gamma} ^{\theta_1}(t)\subset SP_{\beta\gamma} ^{\theta_2}(t)$.
\end{theorem}
\begin{proof}
Since $\liminf\limits_{n\to\infty} \frac{\gamma_n}{T_{\beta\gamma(n)}}>1$, so $\gamma_n>T_{\beta\gamma(n)}$ for all $n$. Also since $0< \beta_n\leq1$, so $$[1,T_{\beta\gamma(n)}]\subseteq[\beta_n,\gamma_n]$$
Let $(\hat{f}_k)= (\hat{f}_k(x))\in N_{\beta\gamma} ^{\theta_1}(t)$ and $\varepsilon>0$. Then for each $x\in[a,b]$
\begin{eqnarray*}
\sum\limits_{k\in[\beta_n,\gamma_n]} t_kd(\hat{f_k}(x),\hat{f}(x)) &\geq&  \sum_{\substack{k\in[\beta_n,\gamma_n]\\ t_kd(\hat{f_k}(x),\hat{f}(x))\geq \varepsilon}} t_kd(\hat{f_k}(x),\hat{f}(x))\\
&=& \varepsilon\Big|\big\{k\in[\beta_n,\gamma_n]:  t_kd(\hat{f_k}(x),\hat{f}(x))\geq \varepsilon\big\}\Big|\\
&\geq& \varepsilon\Big|\big\{k\leq T_{\beta\gamma(n)}:  t_kd(\hat{f_k}(x),\hat{f}(x))\geq \varepsilon\big\}\Big|
\end{eqnarray*}
Diving by $T_{\beta\gamma(n)} ^{\theta_1}$ and applying $\theta_1\leq \theta_2$, we get
\begin{eqnarray*}
\frac{1}{T_{\beta\gamma(n)} ^{\theta_1}} \sum\limits_{k\in[\beta_n,\gamma_n]} t_kd(\hat{f_k}(x),\hat{f}(x))
\geq \frac{\varepsilon}{T_{\beta\gamma(n)} ^{\theta_2}} \Big|\big\{k\leq T_{\beta\gamma(n)}:  t_kd(\hat{f_k}(x),\hat{f}(x))\geq \varepsilon\big\}\Big|
\end{eqnarray*}
Now taking $n\to\infty$ in the above inequality we get $(\hat{f}_k)\in SP_{\beta\gamma} ^{\theta_2}(t)$ i.e. $N_{\beta\gamma} ^{\theta_1}(t)\subset SP_{\beta\gamma} ^{\theta_2}(t)$.\\
To prove the strictness of the above inclusion, consider the example.
\begin{example}
Let $\beta_n=1$, $\gamma_n=n^2$ and $t_n=\frac{1}{5}$. Then $T_{\beta\gamma(n)}=\frac{n^2}{5}$ and hence $\frac{\gamma_n}{T_{\beta\gamma(n)}}=5>1$.\\
Let $(\hat{f}_k):[a,b]\to L(R)$ be a sequence of fuzzy functions defined as
$$\hat{f}_k(x)=\begin{cases}
\begin{cases}
\frac{t}{xk}+1, & -xk\leq t\leq0\\
1-\frac{t}{xk}, & 0\leq t\leq xk
\end{cases}, & \mbox{ when $k$ is a square number}\\
\bar{0}, & \mbox{otherwise}
\end{cases}$$
and so $$t_kd(\hat{f}_k(x),\bar{0})=\begin{cases}
\frac{kx}{5}, & k \mbox{ is a square number}\\
0, & \mbox{otherwise}
\end{cases}$$
Take $\theta_1=\frac{1}{4}$ and $\theta_2=1$. Then for any pre-assigned positive $\varepsilon>0$,
\begin{eqnarray*}
\frac{1}{T_{\beta\gamma(n)} ^{\theta_2}} \Big|\big\{k\leq \frac{n^2}{5}:  t_kd(\hat{f_k}(x),\bar{0})\geq \varepsilon\big\}\Big| &=& \frac{[\sqrt{n^2/5}]}{n^2/5}\\
&\leq& \frac{1}{\sqrt{n^2/5}}=\frac{\sqrt{5}}{n}\to 0\ \mbox{as } n\to\infty
\end{eqnarray*}
So $(\hat{f}_k)\in SP_{\beta\gamma} ^{\theta_2}(t)$, but
\begin{eqnarray*}
\frac{1}{T_{\beta\gamma(n)} ^{\theta_1}} \sum\limits_{k\in[\beta_n,\gamma_n]} t_kd(\hat{f_k}(x),\bar{0}) &=& \frac{1}{(n^2/5)^{1/4}} \sum_{\substack{k\in[1,n^2]\\ k=p^2\ \mbox{for } p\in\mathbb{N}}} \frac{kx}{5}\\
&=& \frac{x}{5^{3/4}} \frac{1}{\sqrt{n}} [1+4+9+\cdots+n^2]\to\infty \ \mbox{as } n\to\infty
\end{eqnarray*}
i.e. $(\hat{f}_k)\notin N_{\beta\gamma} ^{\theta_1}(t)$.
\end{example}
\end{proof}

\begin{theorem}
If $(t_n)$ be a bounded sequence of real numbers such that $\limsup\limits_{n} \frac{\gamma_n}{T_{\beta\gamma(n)} ^{\theta_2}}<\infty$, then $SP_{\beta\gamma} ^{\theta_1}(t)\cap B^F[a,b]\subset N_{\beta\gamma} ^{\theta_2}(t)$ i.e. if a sequence of bounded fuzzy functions which is  weighted $\beta\gamma$ - statistically convergent of order $\theta_1$, then it is absolutely weighted $\beta\gamma$ - summable of order $\theta_2$.
\end{theorem}
\begin{proof}
Let $(\hat{f_k}(x))\in SP_{\beta\gamma} ^{\theta_1}(t)\cap B^F[a,b]$ and since $\limsup\limits_{n} \frac{\gamma_n}{T_{\beta\gamma(n)} ^{\theta_2}}<\infty$, so for some positive $M_1,M_2>0$ such that
\begin{eqnarray*}
\frac{\gamma_n}{T_{\beta\gamma(n)} ^{\theta_2}} &<& M_1\\
\mbox{and }\ t_kd(\hat{f_k}(x),\hat{f}(x)) &<& M_2
\end{eqnarray*}
Let $\varepsilon>0$ be given. For every $x\in[a,b]$ consider the set $$A_n=\{k\leq T_{\beta\gamma(n)}:  t_kd(\hat{f_k}(x),\hat{f}(x))\geq \varepsilon\}$$
Then we have
\begin{eqnarray*}
\sum\limits_{k\in[\beta_n,\gamma_n]} t_kd(\hat{f_k}(x),\hat{f}(x)) &=& \sum\limits_{k\in[\beta_n,\gamma_n]\cap A_n} t_kd(\hat{f_k}(x),\hat{f}(x))\\ && +\sum\limits_{k\in[\beta_n,\gamma_n]\cap A_n ^c} t_kd(\hat{f_k}(x),\hat{f}(x))\\
&\leq& M_2 \Big|\big\{k\leq T_{\beta\gamma(n)}:  t_kd(\hat{f_k}(x),\hat{f}(x))\geq \varepsilon\big\}\Big|\\ && +(\gamma_n+1)\varepsilon\\
\mbox{i.e. } \frac{1}{T_{\beta\gamma(n)} ^{\theta_2}}\sum\limits_{k\in[\beta_n,\gamma_n]} t_kd(\hat{f_k}(x),\hat{f}(x)) &\leq& \frac{M_2}{T_{\beta\gamma(n)} ^{\theta_1}} \Big|\big\{k\leq T_{\beta\gamma(n)}:  t_kd(\hat{f_k}(x),\hat{f}(x))\geq \varepsilon\big\}\Big|\\ && +\Big(M_1+\frac{1}{T_{\beta\gamma(n)} ^{\theta_2}}\Big)\varepsilon
\end{eqnarray*}
Since $\varepsilon>0$ is arbitrary, so the result follows.\\
To show that the inclusion is strict, consider the following example.
\begin{example}
Let $\beta_n=1$, $\gamma_n=n^2$ and $t_n=1+\frac{1}{n}$. Then $\liminf_n t_n=1>0$ and $t_n\leq 2$ for all $n$. Also $T_{\beta\gamma(n)}=n^2+(1+\frac{1}{2}+\cdots+\frac{1}{n})$. Let $\theta_2=1$.
$$\limsup_n \frac{\gamma_n}{T_{\beta\gamma(n)} ^{\theta_2}}=\limsup_n\frac{n^2}{n^2+(1+\frac{1}{2}+\cdots+\frac{1}{n})}\leq 1$$
Define a sequence of fuzzy functions $(\hat{f}_k)$ as
$$\hat{f}_k(x)=\begin{cases}
\begin{cases}
1+\frac{tk}{x}, & -\frac{x}{k}\leq t\leq 0\\
1-\frac{tk}{x}, & 0\leq t\leq \frac{x}{k}
\end{cases}, & k=p^3\ \mbox{for some }p\in\mathbb{N}\\
\bar{0}, & \mbox{otherwise}
\end{cases}$$ and so
$$t_kd(\hat{f}_k(x),\bar{0})=\begin{cases}
\frac{x}{k}+\frac{x}{k^2}, & k=p^3\ \mbox{for some }p\in\mathbb{N}\\
0, & \mbox{otherwise}
\end{cases}$$
\begin{eqnarray*}
\frac{1}{T_{\beta\gamma(n)} ^{\theta_2}}\sum\limits_{k\in[\beta_n,\gamma_n]} t_kd(\hat{f_k}(x),\hat{f}(x)) &=& \frac{1}{n^2+(1+\frac{1}{2}+\cdots+\frac{1}{n})} \sum_{\substack{k\in[1,n^2]\\ k=p^3}} \Big(\frac{x}{k}+\frac{x}{k^2}\Big)\\
&\leq& \frac{x}{n^2} \Big(1+\frac{1}{8}+\frac{1}{27}+\cdots+\frac{1}{[n^{2/3}]}\Big)
\\ && +\frac{x}{n^2}\Big(1+\frac{1}{64}+\cdots+\frac{1}{[n^{4/3}]}\Big)\to0
\end{eqnarray*}
as $n\to\infty$. So, $(\hat{f}_k)\in N_{\beta\gamma} ^{\theta_2}(t)$. Now
\begin{eqnarray*}
\frac{1}{T_{\beta\gamma(n)} ^{\theta_1}}\Big|\big\{k\leq T_{\beta\gamma(n)}:  t_kd(\hat{f_k}(x),\bar{0})\geq \varepsilon\big\}\Big|
&=& \frac{[T_{\beta\gamma(n)} ^{1/3}]}{T_{\beta\gamma(n)} ^{\theta_1}}\\
&\geq& \frac{T_{\beta\gamma(n)} ^{1/3}-1}{T_{\beta\gamma(n)} ^{\theta_1}}
\end{eqnarray*}
choosing $\theta_1<1/3$, it is easy to see that R.H.S tends to $\infty$ as $n\to\infty$. So $(\hat{f}_k)\notin SP_{\beta\gamma} ^{\theta_2}(t)$.
\end{example}
\end{proof}

\begin{remark}
From the above two theorems it is clear that $SP_{\beta\gamma} ^{\theta}(t)$ and $N_{\beta\gamma} ^{\theta}(t)$ are incomaprable class of sequences of fuzzy functions i.e. neither of them is the subset of the other.
\end{remark}

\section{Ordinary weighted $\beta\gamma$-summability of order $\theta$}
Now we define a more general class of summability which is defined as follows:
\begin{definition}
A sequence $(\hat{f_k}(x))$ of fuzzy functions is said to be {\bf ordinary weighted $\beta\gamma$ - summable} to a fuzzy function $\hat{f}$ of order $\theta$ if for every $x\in[a,b]$,
\begin{equation*}
\hat{S}_{\beta\gamma(n)}(x)=\frac{1}{T_{\beta\gamma(n)} ^{\theta}}\sum\limits_{k\in[\beta_n,\gamma_n]} t_k\hat{f}_k(x)\to \hat{f}(x) \ \mbox{ as } n\to\infty
\end{equation*}
i.e. $d\Big(\hat{S}_{\beta\gamma(n)}(x),\hat{f}(x)\Big)\to 0$ as $n\to\infty$ for each $x\in[a,b]$. We denote this convergence as $\bar{N}_{\beta\gamma} (t)$.
\end{definition}

\begin{remark}
As particular cases, we get
\begin{enumerate}
\item For $t_n=1, \beta_n=1,\gamma_n=n$, $\bar{N}_{\beta\gamma} (t)$ is the Cesaro summability discussed in \cite{TC}.
\item For $\beta_n=1,\gamma_n=n$, $\bar{N}_{\beta\gamma} (t)$ reduces to Riesz mean as introduced in \cite{Tr}.
\end{enumerate}
\end{remark}

\begin{theorem}
Let $\hat{f}_k:[a,b]\to L(R)$ and $\hat{g}_k:[a,b]\to L(R)$ be two sequence of fuzzy functions. If $(\hat{f}_k),(\hat{g}_k)\in \bar{N}_{\beta\gamma} (t)$, then $(\hat{f}_k+\hat{g}_k),(c\hat{f}_k)\in \bar{N}_{\beta\gamma} (t)$ for $c\in\mathbb{R}$.
\end{theorem}
\begin{proof}
Proof is easy. So we omit it.
\end{proof}

\begin{theorem}
If a sequence of fuzzy functions is absolutely weighted $\beta\gamma$-summable, then it is weighted $\beta\gamma$ summable i.e. $N_{\beta\gamma} (t)\subset \bar{N}_{\beta\gamma} (t)$.
\end{theorem}
\begin{proof}
Let $\hat{f}_k(x)$ be a sequence of fuzzy functions such that $\Big(\hat{f}_k(x)\Big)\in N_{\beta\gamma} (t)$ for a fuzzy function $\hat{f}(x)$. Now
\begin{eqnarray*}
d(\hat{S}_{\beta\gamma(n)}(x),\hat{f}(x)) &=& d\Big(\frac{1}{T_{\beta\gamma(n)}}\sum\limits_{k\in[\beta_n,\gamma_n]} t_k\hat{f}_k(x),\hat{f}(x)\Big)\\
&=& d\Big(\frac{1}{T_{\beta\gamma(n)}}\sum\limits_{k\in[\beta_n,\gamma_n]} t_k\hat{f}_k(x),\frac{1}{T_{\beta\gamma(n)}}\sum\limits_{k\in[\beta_n,\gamma_n]} t_k \hat{f}(x)\Big)\\
&\leq& \frac{1}{T_{\beta\gamma(n)}}\sum\limits_{k\in[\beta_n,\gamma_n]} t_k d(\hat{f}_k(x),\hat{f}(x))
\end{eqnarray*}
So taking $n\to\infty$ in both sides, we get the inclusion.\\ To show the inclusion is strict, we give the following example.
\begin{example}
Let $t_n=1,\beta_n=1,\gamma_n=n$ and the sequence of fuzzy functions be defined for any $x\in[a,b]$ as
$$\hat{f}_k (x)=\bar{a}_k \ \mbox{ where } a_k=(-1)^{k+1}$$
Then we have
\begin{eqnarray*}
\hat{S}_{\beta\gamma(n)}(x) &=& \frac{1}{n}\sum\limits_{k=1} ^n \overline{(-1)^{k+1}} \\
&=& \begin{cases}
\bar{0}, & \mbox{ if $n$ is even}\\
\frac{1}{n}\bar{1}, & \mbox{ if $n$ is odd}
\end{cases}\\
&\to& \bar{0} \mbox{ as } n\to\infty
\end{eqnarray*}
Then $(\hat{f}_k (x))\in \bar{N}_{\beta\gamma} (t)$. But
\begin{eqnarray*}
\frac{1}{n}\sum\limits_{k\in[1,n]} d(\hat{f}_k(x),\bar{0}) &=& \frac{1}{n}\sum\limits_{k\in[1,n]} 1\\
&=& 1\\
\mbox{i.e. } (\hat{f}_k (x)) &\notin& N_{\beta\gamma} (t)
\end{eqnarray*}
This completes the proof.
\end{example}
\end{proof}

\vspace{1cm}
\noindent Now we will give a generalized version of Tauberian theorem, given by \cite{O}. First we note the following definition.
\begin{definition}\cite{TB}
A sequence of fuzzy numbers $(\hat{u}_n)$ is said to be slowly decreasing if for every $\varepsilon>0$, there exists $n_0=n_0(\varepsilon)$ and $\lambda=\lambda(\varepsilon)>1$ such that for all $n>n_0$
\begin{equation*}
\hat{u}_k \succeq \hat{u}_n-\bar{\varepsilon} \mbox{ whenever } n<k\leq [\lambda n]
\end{equation*}
or equivalently, if for every $\varepsilon>0$, there exists $n_0=n_0(\varepsilon)$ and $\lambda=\lambda(\varepsilon)<1$ such that for all $n>n_0$
\begin{equation*}
\hat{u}_n \succeq \hat{u}_k-\bar{\varepsilon} \mbox{ whenever } [\lambda n]< k\leq n
\end{equation*}
%Where $\lambda_n$ denotes the the integer part of $\lambda n$.
\end{definition}

\noindent Now on the basis of the above definition, we prove our results.

\noindent First we assume that $(\gamma_n)$ be a real sequence taking integer values only and we use the notation $([\lambda\gamma])=([\lambda\gamma_n])$.

\begin{lemma}
For the sequence $T_{\beta\gamma}(n)$, defined above, the following conditions
\begin{eqnarray}
\liminf_{n\to\infty} \frac{T_{\beta([\lambda\gamma])(n)}}{T_{\beta\gamma(n)}} &>& 1 \mbox{ for every } \lambda>1\\
\mbox{and } \liminf_{n\to\infty} \frac{T_{\beta\gamma(n)}}{T_{\beta([\lambda\gamma])(n)}} &>& 1 \mbox{ for every } 0<\lambda<1
\end{eqnarray}
are equivalent.
\end{lemma}
\begin{proof}
Let (2) holds. Then consider two cases : $0<\lambda<1$ and $\lambda>1$.\\
Case (i) : Let $0<\lambda<1$.\\
Set $m=[\lambda\gamma_n]$. Then $\frac{m}{\lambda}=\frac{[\lambda\gamma_n]}{\lambda}\leq \gamma_n$ and thus
\begin{eqnarray*}
T_{\beta([m/\lambda])(n)} &\leq& T_{\beta\gamma(n)}\\
\mbox{and so } \ \frac{T_{\beta\gamma(n)}}{T_{\beta([\lambda\gamma_n])(n)}} &\geq& \frac{T_{\beta([m/\lambda])(n)}}{T_{\beta m(n)}}
\end{eqnarray*}
As the denominators are same and thus
\begin{eqnarray*}
\liminf_{n\to\infty} \frac{T_{\beta\gamma(n)}}{T_{\beta([\lambda\gamma_n])(n)}} &\geq& \liminf_{n\to\infty} \frac{T_{\beta([m/\lambda])(n)}}{T_{\beta m(n)}}>1
\end{eqnarray*}
since $1/\lambda>1$.\\
Case (ii) : Let $\lambda>1$.\\
Choose $\lambda_1$ be such that $1<\lambda_1<\lambda$ and let $m=[\lambda\gamma_n]$. Thus
\begin{eqnarray*}
\gamma_n\leq \frac{[\lambda\gamma_n]-1}{\lambda_1}<\frac{[\lambda\gamma_n]}{\lambda_1}=\frac{m}{\lambda_1}
\end{eqnarray*}
provided $\lambda_1\leq \lambda-1/\gamma_n$, which is the if $n$ is large enough. For such $n$, we have
\begin{eqnarray*}
T_{\beta\gamma(n)} &\leq& T_{\beta([m/\lambda_1])(n)}\\
\mbox{i.e. } \ \frac{T_{\beta([\lambda\gamma_n])(n)}}{T_{\beta\gamma(n)}} &\geq& \frac{T_{\beta m(n)}}{T_{\beta([m/\lambda_1])(n)}}
\end{eqnarray*}
As the numerators are same, so we get
\begin{eqnarray*}
\liminf_{n\to\infty} \frac{T_{\beta([\lambda\gamma_n])(n)}}{T_{\beta\gamma(n)}} &\geq& \liminf_{n\to\infty} \frac{T_{\beta m(n)}}{T_{\beta([m/\lambda_1])(n)}}>1
\end{eqnarray*}
This completes the proof.
\end{proof}

\begin{lemma}
If $\big(T_{\beta\gamma}(n)\big)$ satisfies (2) or (3), then the following conditions
\begin{eqnarray}
\limsup_{n\to\infty} \frac{T_{\beta([\lambda\gamma])(n)}}{T_{\beta([\lambda\gamma])(n)} - T_{\beta\gamma(n)}} &<& \infty \mbox{ for every } \lambda>1\\
\mbox{and } \limsup_{n\to\infty} \frac{T_{\beta([\lambda\gamma])(n)}}{ T_{\beta\gamma(n)}-T_{\beta([\lambda\gamma])(n)}} &<& \infty \mbox{ for every } 0<\lambda<1
\end{eqnarray}
holds accordingly.
\end{lemma}
\begin{proof}
For $\lambda>1$, we have
\begin{eqnarray*}
\limsup_{n\to\infty} \frac{T_{\beta([\lambda\gamma])(n)}}{T_{\beta([\lambda\gamma])(n)}-T_{\beta\gamma(n)}} &=& \limsup_{n\to\infty} \frac{1}{1-\frac{T_{\beta\gamma(n)}}{T_{\beta([\lambda\gamma])(n)}}}\\
&=& \Big\{1- \limsup_{n\to\infty} \frac{T_{\beta\gamma(n)}}{T_{\beta([\lambda\gamma])(n)}}\Big\}^{-1}\\
&=& \Big\{1- \frac{1}{\liminf\limits_{n\to\infty} \frac{T_{\beta([\lambda\gamma])(n)}}{T_{\beta\gamma(n)}}}\Big\}^{-1}<\infty
\end{eqnarray*}
So by condition (2), it follows. Similarly by using (3), we can prove (5).
\end{proof}

\noindent Now on the basis of the above lemmas, we shall try to prove the following theorem.\\
\begin{theorem}
Assume $(\hat{f}_k (x))$ be a sequence of fuzzy functions which is slowly decreasing and (2) is satisfied. Moreover let $(\hat{f}_k (x))$ is $\bar{N}_{\beta\gamma} (t)$ summable to a fuzzy function $\hat{f}(x), x\in[a,b]$. Then the subsequence $\Big(\hat{f}_{\gamma_k} (x)\Big)$ of fuzzy functions converges to $\hat{f}(x)$ for all $x\in[a,b]$.
\end{theorem}
\begin{proof}
Let $(\hat{f}_k (x))$ is $\bar{N}_{\beta\gamma} (t)$ summable to a fuzzy function $\hat{f}(x)$ and is slowly decreasing.\\
Let $\lambda>1$. Then for each $n$, $T_{\beta([\lambda\gamma])(n)}>T_{\beta\gamma(n)}$. So for each $x\in [a,b]$, we have
\begin{eqnarray}
\frac{T_{\beta([\lambda\gamma])(n)}}{T_{\beta([\lambda\gamma])(n)}-T_{\beta\gamma(n)}}\hat{S}_{\beta([\lambda\gamma])(n)}(x)+\hat{S}_{\beta\gamma(n)}(x) &=& \frac{T_{\beta([\lambda\gamma])(n)}}{T_{\beta([\lambda\gamma])(n)}-T_{\beta\gamma(n)}} \frac{1}{T_{\beta([\lambda\gamma])(n)}}\sum\limits_{k\in\big[\beta_n,[\lambda\gamma_n]\big]} t_k\hat{f}_k(x) \nonumber \\ && + \frac{1}{T_{\beta\gamma(n)}}\sum\limits_{k\in[\beta_n,\gamma_n]} t_k\hat{f}_k(x) \nonumber \\
&=& \Big[\frac{1}{T_{\beta([\lambda\gamma])(n)}-T_{\beta\gamma(n)}} + \frac{1}{T_{\beta\gamma(n)}}\Big] \sum\limits_{k\in[\beta_n,\gamma_n]} t_k\hat{f}_k(x) \nonumber\\ && + \frac{1}{T_{\beta([\lambda\gamma])(n)}-T_{\beta\gamma(n)}} \sum\limits_{k\in\big[\gamma_n+1,[\lambda\gamma_n]\big]} t_k\hat{f}_k(x) \nonumber \\
&=& \frac{T_{\beta([\lambda\gamma])(n)}}{T_{\beta([\lambda\gamma])(n)}-T_{\beta\gamma(n)}}\hat{S}_{\beta\gamma(n)}(x) \nonumber \\ && + \frac{1}{T_{\beta([\lambda\gamma])(n)}-T_{\beta\gamma}(n)} \sum\limits_{k\in\big[\gamma_n+1,[\lambda\gamma_n]\big]} t_k\hat{f}_k(x)
\end{eqnarray}
It is given that $\lim\limits_{n\to\infty}\hat{S}_{\beta\gamma(n)}(x) = \hat{f}(x)$. Thus for each $x\in[a,b]$,
\begin{eqnarray}
\hat{f}(x)-\frac{\bar{\varepsilon}}{3} \preceq \hat{S}_{\beta\gamma(n)}(x)\preceq \hat{f}(x)-\frac{\bar{\varepsilon}}{3}
\end{eqnarray}
Also for each $x\in[a,b]$, we have
\begin{eqnarray*}
&&\limsup_{n\to\infty}d\Big(\frac{T_{\beta([\lambda\gamma])(n)}}{T_{\beta([\lambda\gamma])}(n)-T_{\beta\gamma}(n)}\hat{S}_{\beta([\lambda\gamma])(n)}(x),
\frac{T_{\beta([\lambda\gamma])(n)}}{T_{\beta([\lambda\gamma])(n)}-T_{\beta\gamma(n)}}\hat{S}_{\beta\gamma(n)}(x)\Big)\\
&=& \limsup_{n\to\infty} \frac{T_{\beta([\lambda\gamma])(n)}}{T_{\beta([\lambda\gamma])(n)}-T_{\beta\gamma(n)}}d(\hat{S}_{\beta([\lambda\gamma])(n)}(x),\hat{S}_{\beta\gamma(n)}(x))\\
&\leq& \limsup_{n\to\infty} \frac{T_{\beta([\lambda\gamma])}(n)}{T_{\beta([\lambda\gamma])(n)}-T_{\beta\gamma(n)}}\big(d(\hat{S}_{\beta([\lambda\gamma])(n)}(x),\hat{f}(x))+
d(\hat{S}_{\beta\gamma(n)}(x),\hat{f}(x))\big)\\
&\leq& \Big(\limsup_{n\to\infty} \frac{T_{\beta([\lambda\gamma])(n)}}{T_{\beta([\lambda\gamma])(n)}-T_{\beta\gamma}(n)}\Big)\Big(\lim\limits_{n\to\infty} d(\hat{S}_{\beta([\lambda\gamma])(n)}(x),\hat{f}(x))+ \lim\limits_{n\to\infty} d(\hat{S}_{\beta\gamma(n)}(x),\hat{f}(x))\Big)=0
\end{eqnarray*}
since $(\hat{f}_k (x))\in \bar{N}_{\beta\gamma} (t)$ and [4] holds. So for $\varepsilon>0$ and for large $n$ and for $x\in[a,b]$
\begin{eqnarray}
\frac{T_{\beta([\lambda\gamma])(n)}}{T_{\beta([\lambda\gamma])(n)}-T_{\beta\gamma(n)}}\hat{S}_{\beta\gamma(n)}(x)- \frac{\bar{\varepsilon}}{3}
&\preceq& \frac{T_{\beta([\lambda\gamma])(n)}}{T_{\beta([\lambda\gamma])(n)}-T_{\beta\gamma(n)}}\hat{S}_{\beta([\lambda\gamma])(n)}(x) \nonumber \\
&& \preceq \frac{T_{\beta([\lambda\gamma])(n)}}{T_{\beta([\lambda\gamma])}(n)-T_{\beta\gamma(n)}}\hat{S}_{\beta\gamma(n)}(x) + \frac{\bar{\varepsilon}}{3}
\end{eqnarray}
and thus using (7) and (8), we get
\begin{eqnarray*}
\frac{T_{\beta([\lambda\gamma])(n)}}{T_{\beta([\lambda\gamma])(n)}-T_{\beta\gamma(n)}}\hat{S}_{\beta\gamma(n)}(x) + \frac{2\bar{\varepsilon}}{3}+\hat{f}(x)
&\succeq& \frac{T_{\beta([\lambda\gamma])(n)}}{T_{\beta([\lambda\gamma])}(n)-T_{\beta\gamma}(n)}\hat{S}_{\beta([\lambda\gamma])(n)}(x) + \hat{S}_{\beta\gamma(n)}(x)\\
&=&  \frac{T_{\beta([\lambda\gamma])(n)}}{T_{\beta([\lambda\gamma])(n)}-T_{\beta\gamma(n)}}\hat{S}_{\beta\gamma(n)}(x)\\ && + \frac{1}{T_{\beta([\lambda\gamma])(n)}-T_{\beta\gamma(n)}}\sum\limits_{k\in\big[\gamma_n+1,[\lambda\gamma_n]\big]} t_k\hat{f}_k(x) \\ && \mbox{(from [6])}\\
&\succeq& \frac{T_{\beta([\lambda\gamma])(n)}}{T_{\beta([\lambda\gamma])(n)}-T_{\beta\gamma(n)}}\hat{S}_{\beta\gamma(n)}(x) \\ && +  \frac{1}{T_{\beta([\lambda\gamma])(n)}-T_{\beta\gamma(n)}}\sum\limits_{k\in\big[\gamma_n+1,[\lambda\gamma_n]\big]} t_k\Big(\hat{f}_{\gamma_n}(x)-\frac{\bar{\varepsilon}}{3}\Big)\\
&& \mbox{(since $(\hat{f}_n(x))$ is slowly decreasing)}\\
&=& \frac{T_{\beta([\lambda\gamma])(n)}}{T_{\beta([\lambda\gamma])(n)}-T_{\beta\gamma(n)}}\hat{S}_{\beta\gamma(n)}(x) +  \hat{f}_{\gamma_n}(x)-\frac{\bar{\varepsilon}}{3}
\end{eqnarray*}
Thus we get
\begin{eqnarray}
\hat{f}(x)+\bar{\varepsilon} &\succeq& \hat{f}_{\gamma_n}(x)
\end{eqnarray}
On the other case let $0<\lambda<1$. So for each $n$, $T_{\beta([\lambda\gamma])(n)}<T_{\beta\gamma(n)}$.
\begin{eqnarray}
\frac{T_{\beta([\lambda\gamma])(n)}}{T_{\beta\gamma(n)}-T_{\beta([\lambda\gamma])(n)}} \hat{S}_{\beta([\lambda\gamma])(n)}(x)
&+& \frac{1}{T_{\beta\gamma(n)} - T_{\beta([\lambda\gamma])(n)}} \sum\limits_{k\in\big[[\lambda\gamma_n]+1,\gamma_n\big]} t_k\hat{f}_k(x) \nonumber \\
&=& \frac{1}{T_{\beta\gamma(n)} - T_{\beta([\lambda\gamma])(n)}} \Big[\sum\limits_{k\in\big[\beta_n,[\lambda\gamma_n]\big]} t_k\hat{f}_k(x)\nonumber\\
&& +\sum\limits_{k\in[[\lambda\gamma_n]+1,\gamma_n]} t_k\hat{f}_k(x)\Big] \nonumber\\
&=& \frac{1}{T_{\beta\gamma(n)} - T_{\beta([\lambda\gamma])(n)}} \sum\limits_{k\in[\beta_n,\gamma_n]} t_k\hat{f}_k(x)\nonumber\\
&=& \frac{T_{\beta\gamma(n)}}{T_{\beta\gamma(n)} - T_{\beta([\lambda\gamma])(n)}} \hat{S}_{\beta\gamma(n)}(x)\nonumber\\
&=& \frac{T_{\beta([\lambda\gamma])(n)}}{T_{\beta\gamma(n)} - T_{\beta([\lambda\gamma])(n)}} \hat{S}_{\beta\gamma(n)}(x) \nonumber\\ &&+\hat{S}_{\beta\gamma(n)}(x)
\end{eqnarray}
Also proceeding as before and with the help of [5], we get
\begin{eqnarray}
\frac{T_{\beta([\lambda\gamma])(n)}}{T_{\beta\gamma(n)}-T_{\beta([\lambda\gamma])(n)}}\hat{S}_{\beta([\lambda\gamma])(n)}(x)- \frac{\bar{\varepsilon}}{3}
&\preceq& \frac{T_{\beta([\lambda\gamma])(n)}}{T_{\beta([\lambda\gamma])}(n)-T_{\beta\gamma(n)}}\hat{S}_{\beta\gamma(n)}(x) \nonumber \\
&\preceq& \frac{T_{\beta([\lambda\gamma])(n)}}{T_{\beta\gamma(n)}-T_{\beta([\lambda\gamma])(n)}}\hat{S}_{\beta([\lambda\gamma])(n)}(x) \nonumber\\
&& + \frac{\bar{\varepsilon}}{3}
\end{eqnarray}
Thus for each $x\in[a,b]$
\begin{eqnarray*}
&& \frac{T_{\beta([\lambda\gamma])(n)}}{T_{\beta\gamma(n)}-T_{\beta([\lambda\gamma])(n)}}\hat{S}_{\beta([\lambda\gamma])(n)}(x)+\hat{f}(x)- \frac{2\bar{\varepsilon}}{3}\\
&\preceq& \frac{T_{\beta([\lambda\gamma])(n)}}{T_{\beta\gamma(n)} - T_{\beta([\lambda\gamma])(n)}} \hat{S}_{\beta\gamma(n)}(x) +\hat{S}_{\beta\gamma(n)}(x)\\
&& \mbox{(from [7] and [11])}\\
&=& \frac{T_{\beta([\lambda\gamma])(n)}}{T_{\beta\gamma(n)}-T_{\beta([\lambda\gamma])(n)}}\hat{S}_{\beta([\lambda\gamma])(n)}(x) + \frac{1}{T_{\beta\gamma(n)} - T_{\beta([\lambda\gamma])(n)}} \sum\limits_{k\in\big[[\lambda\gamma_n]+1,\gamma_n\big]} t_k\hat{f}_k(x)\\
&& \mbox{(follows from [10])}\\
&\preceq& \frac{T_{\beta([\lambda\gamma])(n)}}{T_{\beta\gamma(n)}-T_{\beta([\lambda\gamma])(n)}}\hat{S}_{\beta([\lambda\gamma])(n)}(x) + \hat{f}_{\lambda_n}(x) +\frac{\bar{\varepsilon}}{3}
\end{eqnarray*}
since $\hat{f}_n(x)$ is slowly decreasing. Thus
\begin{eqnarray}
\hat{f}(x)- \bar{\varepsilon} \preceq \hat{f}_{\gamma_n}(x)
\end{eqnarray}
Combining [9] and [12], we get
\begin{eqnarray*}
\hat{f}(x)- \bar{\varepsilon} \preceq \hat{f}_{\gamma_n}(x) \preceq \hat{f}(x)+ \bar{\varepsilon}, \ \ x\in[a,b]
\end{eqnarray*}
So $\Big(\hat{f}_{\gamma_k} (x)\Big)$ converges to $\hat{f}(x)$ for all $x\in[a,b]$.
\end{proof}

\begin{remark}
In case of $\gamma_n=n$ and $\beta_n=1$, we get the Tauberian theorem as in \cite{O} in case of fuzzy numbers. So the above theorem is definitely a generalization of the previous versions of the Tauberian theorem in the literature.
\end{remark}

\end{document}